\RequirePackage[l2tabu, orthodox]{nag}
\documentclass[
11pt]{article}

\usepackage[inner=3.5cm,outer=3.5cm,top=2.5cm,bottom=3.2cm,headsep=0.5cm]{geometry}

\usepackage{subcaption,overpic}

\usepackage [textsize=small]{todonotes}

\usepackage{tikz,tkz-graph,bm}

\usepackage[final]{microtype}
\makeatletter
\def\MT@register@subst@font{\MT@exp@one@n\MT@in@clist\font@name\MT@font@list
   \ifMT@inlist@\else\xdef\MT@font@list{\MT@font@list\font@name,}\fi}
\makeatother 

\usepackage[small]{titlesec}

\usepackage{mparhack}

\usepackage{amsmath, amssymb, amsfonts, amsthm, mathtools}
\usepackage{colonequals}

\usepackage{fixmath}

\usepackage[latin1]{inputenc}
\usepackage[T1]{fontenc}

\usepackage[draft=false,pdftex]{hyperref}
\hypersetup{
    unicode=false,          
    pdftoolbar=true,        
    pdfmenubar=true,        
    pdffitwindow=false,     
    pdfstartview={FitH},    
    pdftitle={Almost-equidistant sets},    
    pdfnewwindow=false,      
    colorlinks=true,       
    linkcolor=black,          
    citecolor=black,        
    filecolor=black,      
    urlcolor=black           
}

\theoremstyle{plain}
\newtheorem{theorem}{Theorem}
\newtheorem{lemma}[theorem]{Lemma}
\newtheorem{corollary}[theorem]{Corollary}
\newtheorem{proposition}[theorem]{Proposition}
\newtheorem{conjecture}{Conjecture}

\theoremstyle{definition}

\newcommand{\setbuilder}[2]{\left\{#1\;\middle|\;#2\right\}}

\newcommand{\rank}{\operatorname{rank}}

\newcommand{\epsi}{\varepsilon}

\newcommand{\norm}[1]{\left\lVert#1\right\rVert}
\newcommand{\ipr}[2]{\left\langle #1, #2 \right\rangle}

\newcommand{\card}[1]{\left\lvert#1\right\rvert}
\newcommand{\numbersystem}[1]{\mathbb{#1}}
\newcommand{\bN}{\numbersystem{N}}
\newcommand{\bR}{\numbersystem{R}}
\DeclareMathOperator{\aff}{aff}
\DeclareMathOperator{\lin}{span}

\title{Almost-equidistant sets\thanks{M. Balko has received funding from European Research Council (ERC) under the European Union's Horizon 2020 research and innovation programme under grant agreement no.~678765.
The authors all acknowledge the support of the ERC Advanced Research Grant no.~267165 (DISCONV).
M.~Balko and P.~Valtr acknowledge support of the grant P202/12/G061 of the Czech Science Foundation (GA\v CR). 
M.~Scheucher acknowledges support from TBK Automatisierung und Messtechnik GmbH.
}}
\author{Martin Balko\\
\small Department of Applied Mathematics and Institute for Theoretical Computer Science,\\[-0.8ex]
\small Faculty of Mathematics and Physics, \\[-0.8ex]
\small Charles University, Prague, Czech Republic\\
\small\texttt{balko@kam.mff.cuni.cz}\\
\small Department of Computer Science, Faculty of Natural Sciences, \\[-0.8ex]
\small Ben-Gurion University of the Negev, Beer~Sheva, Israel\\
\and 
Attila P\'or\\
\small Department of Mathematics,\\[-0.8ex]
\small Western Kentucky University, \\[-0.8ex]
\small Bowling Green, KY 42101\\
\small\texttt{attila.por@wku.edu}\\
\and
Manfred Scheucher\\
\small  Institut f\"ur Mathematik,  \\[-0.8ex]
\small  Technische Universit\"at Berlin,  \\[-0.8ex]
\small  Germany \\
\small\texttt{scheucher@math.tu-berlin.de}\\
\and
Konrad Swanepoel\\ 
\small Department of Mathematics,  \\[-0.8ex]
\small London School of Economics and Political Science,  \\[-0.8ex]
\small London, United Kingdom \\
\small\texttt{k.swanepoel@lse.ac.uk}\\
\and
Pavel Valtr\\
\small Department of Applied Mathematics and Institute for Theoretical Computer Science,\\[-0.8ex]
\small Faculty of Mathematics and Physics, \\[-0.8ex]
\small Charles University, Prague, Czech Republic
}
\date{}

\begin{document}
\maketitle

\begin{abstract}
For a positive integer $d$, a set of points in $d$-dimensional Euclidean space is called \emph{almost-equidistant} if for any three points from the set, some two are at unit distance.
Let $f(d)$ denote the largest size of an almost-equidistant set in $d$-space.

It is known that $f(2)=7$, $f(3)=10$, and that the extremal almost-equidistant sets are unique.
We give independent, computer-assisted proofs of these statements.
It is also known that $f(5) \ge 16$.
We further show that $12\leq f(4)\leq 13$, $f(5)\leq 20$, $18\leq f(6)\leq 26$, $20\leq f(7)\leq 34$, and $f(9)\geq f(8)\geq 24$. 
Up to dimension $7$, our work is based on various computer searches, and in dimensions $6$ to $9$, we give constructions based on the known construction for $d=5$.

For every dimension $d \ge 3$, we give an example of an almost-equidistant set of $2d+4$ points in the $d$-space and we prove the asymptotic upper bound $f(d) \le O(d^{3/2})$.
\end{abstract}

\section{Introduction and our results}

For a positive integer $d$, we denote the $d$-dimensional Euclidean space by $\bR^d$.
A set $V$ of (distinct) points in $\bR^d$ is called \emph{almost equidistant} if among any three of them, some pair is at distance $1$.
Let $f(d)$ be the maximum size of an almost-equidistant set in $\bR^d$.
For example, the vertex set of the well-known Moser spindle (Figure~\ref{fig1}) is an almost-equidistant set of~$7$ points in the plane and thus $f(2) \ge 7$.

In this paper we study the growth rate of the function $f$.
 We first consider the case when the dimension $d$ is small and give some 
almost tight estimates on $f(d)$ for $d \le 9$.
Then we turn to higher dimensions and show $2d+4 \le f(d) \le O(d^{3/2})$.
We also discuss some possible generalisations of the problem.
\begin{figure}
\centering
\includegraphics[scale=0.5]{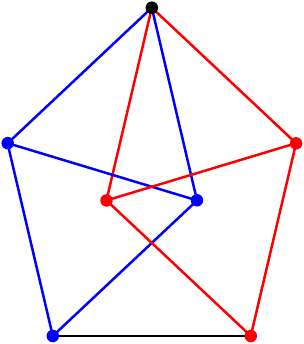}
\caption{The Moser spindle.}
\label{fig1}
\end{figure}

\subsection{Our results in low dimensions}

It is trivial 
that $f(1)=4$ and that, up to congruence, there is a unique almost-equidistant set on $4$ points in $\bR$.
Bezdek, Nasz\'odi, and Visy~\cite{Bezdek-Naszodi-Visy} showed that an almost-equidistant set in the plane has at most $7$ points.
Talata (personal communication) showed in 2007 that there is a unique extremal set.
We give a simple, computer-assisted proof of this result.

\begin{theorem}[Talata, 2007]
\label{thm:2d}
The largest number of points in an almost-equidistant set in~$\bR^2$ is~$7$, that is, $f(2)=7$.
Moreover, up to congruence, there is only one planar almost-equidistant set with $7$ points, namely the Moser spindle.
\end{theorem}

Figure~\ref{fig2} shows an example of an almost-equidistant set of $10$ points in $\bR^3$.
\begin{figure}
\centering
\includegraphics[scale=0.5]{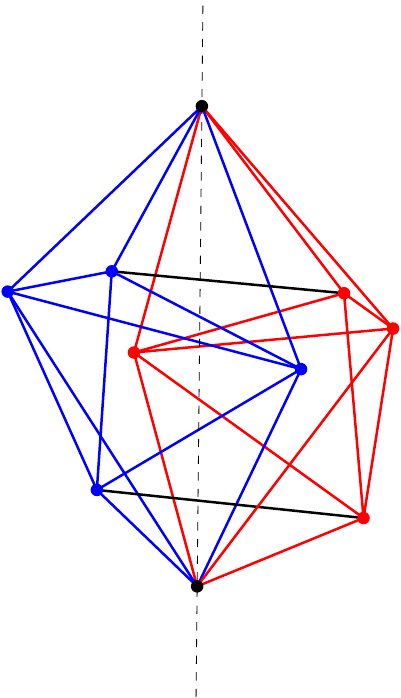}
\caption{An almost-equidistant set in $\bR^3$ made of two biaugmented tetrahedra.}
\label{fig2}
\end{figure}
It is made by taking a so-called \emph{biaugmented tetrahedron}, which is a non-convex polytope formed by gluing three unit tetrahedra together at faces, and rotating a copy of it along the axis through the two simple vertices so that two additional unit-distance edges are created.
This unit-distance graph is used in a paper of Nechushtan~\cite{Nechushtan} to show that the chromatic number of $\bR^3$ is at least $6$.
Gy\"orey~\cite{Gyorey2004} showed, by an elaborate case analysis, that this is the unique largest almost-equidistant set in dimension~$3$.
We again give an independent, computer-assisted  proof.

\begin{theorem}[Gy\"orey~\cite{Gyorey2004}]
\label{thm:3d}
The largest number of points in an almost-equidistant set in~$\bR^3$ is~$10$, that is, $f(3)=10$.
Moreover, up to congruence, there is only one almost-equidistant set in $\bR^3$ with $10$ points.
\end{theorem}

In dimension $4$, we have only been able to obtain the following bounds.

\begin{theorem}
\label{thm:4d}
The largest number of points in an almost-equidistant set in $\bR^4$ is either $12$ or $13$, that is, $f(4) \in \{12,13\}$.
\end{theorem}

The lower bound comes from a generalization of the example in Figure~\ref{fig2}; see also Theorem~\ref{thm:lowerbound}.
The proofs of the upper bounds in the above theorems are computer assisted.
Based on some numerical work to find approximate realisations of graphs, we believe, but cannot prove rigorously, that there does not exist an almost-equidistant set of $13$ points in $\bR^4$.
\begin{conjecture}
The largest number of points in an almost-equidistant set in $\bR^4$ is $12$, that is, $f(4) =12$.
\end{conjecture}

In dimension $5$, Larman and Rogers~\cite{Larman-Rogers} showed that $f(5)\geq 16$ by a construction based on the so-called Clebsch graph.
In dimensions $6$ to $9$, we use their construction to obtain lower bounds that are stronger than the lower bound $2d+4$ stated below in Theorem~\ref{thm:lowerbound}. 
We again complement this with some computer-assisted upper bounds.
\begin{theorem}
\label{thm:5d}
The largest number of points in an almost-equidistant set in $\bR^5$, $\bR^6$, $\bR^7$, $\bR^8$ and $\bR^9$ 
satisfy the following:
$16\leq f(5)\leq 20$, $18\leq f(6)\leq 26$, $20\leq f(7)\leq 34$, $24\leq f(8)\leq 41$, and $24\leq f(9)\leq 49$.
\end{theorem}

The \emph{unit-distance graph} of an almost-equidistant point set $P$ in $\bR^d$ is the graph obtained from $P$ by letting $P$ be its vertex set and by placing an edge between pairs of points at unit distance.

For every $d \in \bN$, a unit-distance graph in $\bR^d$ does not contain $K_{d+2}$ (see Corollary~\ref{corollary:clique}) and the complement of the unit-distance graph of an almost-equidistant set is triangle-free.
Thus we have $f(d)\leq R(d+2,3)-1$, where $R(d+2,3)$ is the \emph{Ramsey number} of $K_{d+2}$ and $K_3$,
that is, the smallest positive integer $N$ such that for every graph $G$ on $N$ vertices there is a copy of $K_{d+2}$ in~$G$ or a copy of $K_3$ in the complement of~$G$.

Ajtai, Koml\'{o}s, and Szemer\'{e}di~\cite{ajKomSze80} showed $R(d+2,3) \leq O(d^2/\log{d})$ and this bound is known to be tight~\cite{Kim95}.
We thus have an upper bound $f(d) \leq O(d^2/\log{d})$, which, as we show below, is not tight.
For small values of $d$ where the Ramsey number $R(d+2,3)$ is known or has a reasonable upper bound, we obtain an upper bound for~$f(d)$.
In particular, we get $f(5) \le 22$, $f(6) \le 27$, $f(7) \le 35$, $f(8) \le 41$, and $f(9) \leq 49$~\cite{radz94}.
For $d \in \{5,6,7\}$, we slightly improve these estimates to the bounds from Theorem~\ref{thm:5d} using our computer-assisted approach.

\begin{table}
\centering
\begin{tabular}{l|rrrrrrrrrc}
Dimension $d$ & $1$ & $2$ & $3$ & $4$ & $5$ & $6$ & $7$ & $8$ & 9 & $d\geq 9$\\
\hline
Lower bounds on $f(d)$ & $4$ & $7$ & $10$ & $12$ & $16$ & $18$ & $20$ & $24$ & $24$ & $2d+4$ \\
Upper bounds on $f(d)$ & $4$ & $7$ & $10$ & $13$ & $20$ & $26$ & $34$ & $41$ & $49$ & $4(d^{3/2}+\sqrt{d})$
\end{tabular}
\caption{Lower and upper bounds on the largest size of an almost-equidistant set in~$\bR^d$.}
\label{table1}
\end{table}

\subsection{Our results in higher dimensions}

We now turn to higher dimensions.
The obvious generalization of the Moser spindle gives an example of an almost-equidistant set of $2d+3$ points in $\bR^d$.
The next theorem improves this by~$1$.
It is a generalization of the almost-equidistant set on $10$ points in $\bR^3$ from Figure~\ref{fig2}.

\begin{theorem}\label{thm:lowerbound}
In each dimension $d\geq 3$, there is an almost-equidistant set in $\bR^d$ with $2d+4$ points.
\end{theorem}

Rosenfeld~\cite{Rosenfeld} showed that an almost-equidistant set on a sphere in $\bR^d$ of radius $1/\sqrt{2}$ has size at most $2d$, which is best possible.
Rosenfeld's proof, which uses linear algebra, was adapted by Bezdek and Langi~\cite{Bezdek-Langi} to spheres of other radii.
They showed that an almost-equidistant set on a sphere in $\bR^d$ of radius $\leq1/\sqrt{2}$ has at most $2d+2$ elements, which is attained by the union of two $d$-simplices inscribed in the same sphere.

Pudl\'ak~\cite{Pudlak} and Deaett~\cite{Deaett} gave simpler proofs of Rosenfeld's result.
Our final result is an asymptotic upper bound for the size of an almost-equidistant set, based on Deaett's proof~\cite{Deaett}.

\begin{theorem}\label{thm1}
An almost-equidistant set of points in $\bR^d$ has cardinality $O(d^{3/2})$.
\end{theorem}

We note that Polyanskii~\cite{Polyanskii} recently found an upper bound of $O(d^{13/9})$ for the size of an almost-equidistant set in~$\bR^d$ and Kupavskii, Mustafa, and Swanepoel~\cite{KMS} improved this to $O(d^{4/3})$.
Both papers use ideas from our proof of Theorem~\ref{thm1}.

In this paper, we use $\norm{v}$ to denote the Euclidean norm of a vector $v$ from $\bR^d$.
For a subset $S$ of $\bR^d$, we use $\lin(S)$ and $\aff(S)$ to denote the linear hull and the affine hull of~$S$, respectively.

The proofs of Theorems~\ref{thm:lowerbound} and~\ref{thm1} are in Section~\ref{section:highd}.
Theorems~\ref{thm:2d} to~\ref{thm:5d} are proved in Section~\ref{section:lowd}.
In Section~\ref{section:general}, we discuss a possible generalization of the problem of determining the function $f$.

\section{High dimensions}
\label{section:highd}

In this section, we first prove Theorem~\ref{thm:lowerbound} by constructing, for every integer $d \ge 3$, an almost-equidistant set in $\bR^d$ with $2d+4$ points.
In the rest of the section, we prove Theorem~\ref{thm1} by showing the upper bound $f(d) \le O(d^{3/2})$.

\subsection{Proof of Theorem~\ref{thm:lowerbound}}

Consider an equilateral $d$-simplex $\triangle$ with vertex set $S=\{x_0,\dots,x_d\}$.
Let $x_i'$ be the reflection of $x_i$ in the hyperplane through the facet of $\triangle$ not containing~$x_i$.
We will show that there exists an isometry $\rho$ of $\bR^d$ that fixes the line determined by $x_0'$ and $x_1'$ such that the distances $\norm{x_0-\rho(x_0)}$ and $\norm{x_1-\rho(x_1)}$ are both~$1$, and such that $\{x_0,\dots,x_d\}$ is disjoint from $\{\rho(x_0),\dots,\rho(x_d)\}$.

With such an isometry $\rho$, the set $R:=S\cup\rho(S)\cup\{x_0',x_1'\}$ clearly contains $2d+4$ distinct points.
We next show that $R$ is almost equidistant.
Suppose for contradiction there is a subset of $R$ with three points and with no pair of points at distance~$1$.
Since $S$ and $\rho(S)$ form cliques in the unit-distance graph of $R$ and every point from $S \cup \rho(S)$ is at unit distance from $x'_0$ or $x'_1$, this subset is necessarily $\{x_i,\rho(x_j), x_k'\}$ for some $i,j\in\{0,\dots,d\}$ and $k\in\{0,1\}$.
Since $x_k'$ is at distance $1$ from all other points in $S\cup\rho(S)$ except $x_k$ and $\rho(x_k)$, we obtain $i=k$ and $j=k$.
However, then the points $x_k$ and $\rho(x_k)$ are at distance 1, as the isometry $\rho$ is chosen so that $\norm{x_k-\rho(x_k)}=1$, contradicting the choice of the $3$-point subset.

It remains to show that there exists an isometry $\rho$ as described above.
Let $c:=\frac12(x_0+x_1)$, and assume without loss of generality that $\frac12(x_0'+x_1')$ is the origin $o$.
Let $H$ be the $(d-1)$-dimensional subspace through $o$ with normal $x_1-x_0$.
Note that $c,x_2,\dots,x_d\in H$.
Let $V$ be any $2$-dimensional subspace containing $o$ and $c$.
(For instance, we can let $V$ be the linear span of $c$ and $x_2$.)
Let $\pi$ be the orthogonal projection of $\mathbb{R}^d$ onto~$V$.
Then $\pi(x_0)=\pi(x_1)=c$ and $\pi(x_0')=\pi(x_1')=o$.
Let $\rho_V$ be a rotation of $V$ around $o$ such that $\norm{c-\rho_V(c)}=1$.
Define $\rho\colon\bR^d\to\bR^d$ by $\rho(x):=\rho_V(\pi(x))+(x-\pi(x))$.
Then $\rho$ is an isometry that fixes the orthogonal complement of $V$ (including $x_0'$ and $x_1'$) and moves $x_0$ and $x_1$ by a distance of $1$.
Since $x_0,x_1\notin H$, it follows that $\rho(x_0),\rho(x_1)\notin H$, hence $\rho(x_0),\rho(x_1)$ do not coincide with any of $x_2,\dots,x_d$.
Similarly, $x_0,x_1$ do not coincide with any of $\rho(x_2),\dots,\rho(x_d)$.
It remains to show that $\{x_2,\dots,x_d\}$ is disjoint from $\{\rho(x_2),\dots,\rho(x_d)\}$.
It is sufficient to prove that $\norm{x_i-\rho(x_i)}<1$ for all $i=2,\dots,d$, since $\norm{x_i-x_j}=1$ for all distinct $i$ and $j$, and $\rho$ does not fix any of $x_2,\dots,x_d$.
We first calculate that
\begin{equation}\label{lengths}
\norm{c-o}=\sqrt{1-\frac{1}{d^2}}\quad\text{and}\quad\norm{x_i-o}=\sqrt{\frac{3}{4}-\frac{1}{d}-\frac{1}{d^2}} < \norm{c-o}.
\end{equation}
Note that \[x_0'=\frac{2}{d}(x_1+\dots+x_d)-x_0 = \frac{2}{d}\sum_{i=0}^d x_i - \left(1+\frac{2}{d}\right)x_0,\] and similarly, \[ x_1' = \frac{2}{d}\sum_{i=0}^d x_i - \left(1+\frac{2}{d}\right)x_1.\]
It follows that \[ c = c-o =\frac12(x_0+x_1)-\frac12(x_0'+x_1') = \left(1-\frac{1}{d}\right)(x_0+x_1) - \frac{2}{d}(x_2+\dots+x_d).\]
We can embed $\bR^d$ isometrically into the hyperplane $\setbuilder{(\lambda_0,\dots,\lambda_d)}{\lambda_0+\dots+\lambda_d=1}$ of~$\bR^{d+1}$ by sending $x_i$ to $\frac{1}{\sqrt{2}}e_i\in\bR^{d+1}$, $i=0,\dots,d$.
It follows that \[\norm{c-o}=\frac{1}{\sqrt{2}}\norm{\left(1-\frac{1}{d},1-\frac{1}{d},-\frac{2}{d},\dots,-\frac{2}{d}\right)} = \sqrt{1-\frac{1}{d^2}},\]
which is the first half of \eqref{lengths}.
Similarly, it follows that
\[ \norm{x_i-o} = \frac{1}{\sqrt{2}}\norm{\left(\frac{1}{2}-\frac{1}{d},\frac{1}{2}-\frac{1}{d},1-\frac{2}{d},-\frac{2}{d},\dots,-\frac{2}{d}\right)} = \sqrt{\frac{3}{4}-\frac{1}{d}-\frac{1}{d^2}},\]
which establishes the second half of \eqref{lengths}.
Since the isometry $\rho$ is a rotation in $V$, moving each point in $V$ at distance $\sqrt{1-\frac{1}{d^2}}$ from $o$ by a distance of $1$, it will move each point in $\bR^d$ at distance less than $\sqrt{1-\frac{1}{d^2}}$ from $o$ by a distance less than $1$.
Therefore, $x_2,\dots,x_d$ are all moved by a distance less than $1$, and the proof is finished.

\subsection{Proof of Theorem~\ref{thm1}}

As a first step towards this proof, we show the following lemma, whose statement is illustrated in Figure~\ref{fig-sphere}.
This lemma is also used later in Section~\ref{section:lowd}.

\begin{lemma}
\label{lemma:sphere}
For $d,k\in \bN$, let $C$ be a set of $k$ points in $\bR^d$ such that the distance between any two of them is $1$.
Let $c \colonequals \frac{1}{k} \sum_{p \in C} p$ be the centroid of $C$ and let $A \colonequals \lin(C-c)$.
Then the set of points equidistant from all points of $C$ is the affine space $c+A^\perp$ orthogonal to $A$ and passing through $c$.
Furthermore, the intersection of all unit spheres centred at the points in $C$ is the $(d-k)$-dimensional sphere of radius $\sqrt{(k+1)/(2k)}$ centred at $c$ and contained in $c+A^{\perp}$.
\end{lemma}

\begin{figure}[ht]
\centering
\includegraphics[scale=1.1]{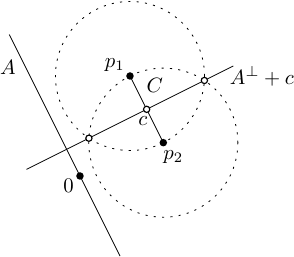}
\caption{An illustration of the statement of Lemma~\ref{lemma:sphere} for $d=k=2$.
The points $p_1$ and $p_2$ are at distance $1$.
The intersection of the unit spheres centred at $p_1$ and $p_2$ is the $0$-dimensional sphere of radius $\sqrt{(k+1)/(2k)}=\sqrt{3}/2$ centred at $c$.}
\label{fig-sphere}
\end{figure}

\begin{proof}
Let $C=\{p_1,\dots,p_k\}$.
First, we show that a point $x$ from $\bR^d$ lies in $c+A^{\perp}$ if and only if $\norm{x-p_1}=\cdots=\norm{x-p_k}$, which gives the first part of the lemma.
For every $x \in \bR^d$ and each $i \in \{1,\dots,k\}$, we have
\begin{align*}
\norm{x-p_i}^2 &= \norm{(x-c)-(p_i-c)}^2\\
&= \norm{x-c}^2 + \norm{p_i-c}^2 -2\ipr{x-c}{p_i-c}.
\end{align*}
Let $\triangle$ be the standard $(k-1)$-simplex with vertices $(1,0,\dots,0),\dots,(0,\dots,0,1)$ embedded in $\bR^k$.
Note that $\sqrt{2}\norm{p_i-c}$ is the distance between a vertex of $\triangle$ and its centroid $(1/k,\dots,1/k)$, which is $\sqrt{(k-1)/k}$. 
Therefore $\norm{p_i-c} = \sqrt{(k-1)/(2k)}$ and, in particular,
\begin{equation}
\label{eq-distance}
\norm{x-p_i}^2 = \norm{x-c}^2 + \frac{k-1}{2k} -2\ipr{x-c}{p_i-c}
\end{equation}
for every $i \in \{1,\dots,k\}$.

Now, assume that $x \in c+A^{\perp}$.
Then, since $p_i-c \in A$ for every $i \in \{1,\dots,k\}$, we have $\ipr{x-c}{p_i-c} = 0$.
By~\eqref{eq-distance}, we obtain
\[\norm{x-p_i}^2 = \norm{x-c}^2 + \frac{k-1}{2k}\]
for every $i \in \{1,\dots,k\}$ and thus $\norm{x-p_1}=\cdots=\norm{x-p_k}$.

On the other hand, if there is a $\gamma \in \bR$ such that $\gamma = \norm{x-p_i}$ for every $i \in \{1,\dots,k\}$, then the equality~\eqref{eq-distance} gives
\[\gamma^2 = \norm{x-c}^2 + \frac{k-1}{2k} -2\ipr{x-c}{p_i-c}\]
for every $i \in \{1,\dots,k\}$.
Setting $\eta \colonequals \bigl(\norm{x-c}^2 + (k-1)/(2k) - \gamma^2\bigr)/2$, we have $\ipr{x-c}{p_1-c}=\dots=\ipr{x-c}{p_k-c}=\eta$.
Using this fact and the expression of $c$, we obtain
\[0=\ipr{x-c}{o} = \ipr{x-c}{c-c} = \ipr{x-c}{\frac{1}{k}\sum_{i=1}^k p_i-c} = \frac{1}{k}\sum_{i=1}^k \ipr{x-c}{p_i-c} =\eta,\]
where $o$ denotes the origin in $\bR^d$.
Thus $\ipr{x-c}{p_i-c}=\eta=0$ for every $i \in \{1,\dots,k\}$ and, since every element of $A$ is a linear combination of elements from $C-c$, we have $x \in c+A^{\perp}$.

For the second part of the lemma, let $S$ be the intersection of all unit spheres centred at the points in $C$.
By the first part of the lemma, we know that
\[S= \setbuilder{x \in c+A^{\perp}}{1 = \norm{x-c}^2 + \frac{k-1}{2k}}.\]
Since $A$ is a $(k-1)$-dimensional subspace, $c+A^{\perp}$ is a $(d-k+1)$-dimensional affine subspace of $\bR^d$, hence $S$ is the $(d-k)$-dimensional sphere of radius $\sqrt{(k+1)/(2k)}$ centred at $c$ and contained in $c+A^{\perp}$.
\end{proof}

\begin{corollary}
\label{corollary:clique}
For $d \in \bN$, every subset of $\bR^d$ contains at most $d+1$ points that are pairwise at unit distance.
\end{corollary}
\begin{proof}
Lemma~\ref{lemma:sphere} applied to a set $C$ of $d$ points in $\bR^d$ with all pairs of points at unit distance implies that the set of points that are at unit distance from all points in $C$ lies on a $0$-dimensional sphere of diameter $2\sqrt{(d+1)/(2d)} \neq 1$.
\end{proof}

The following lemma is a well-known result that bounds the rank of a square matrix from below in terms of the entries of the matrix~\cite{Alon, Deaett, Pudlak}.

\begin{lemma}
\label{lemma:rank}
Let $A=[a_{i,j}]$ be a non-zero symmetric $m\times m$ matrix with real entries.
Then \[\rank A \geq \Bigl(\sum\limits_{i=1}^m a_{i,i}\Bigr)^2 / \sum\limits_{i=1}^m\sum\limits_{j=1}^m a_{i,j}^2.\]
\end{lemma}

The last lemma before the proof of Theorem~\ref{thm1} can be proved by a calculation, using its assumption that 
the vectors $v_i$ have pairwise inner products $\epsi$, so they differ from an orthogonal set by some skewing.

\begin{lemma}\label{lemma:calc}
For $n,t \in \bN$ with $t \leq n$, let $w_1,\dots,w_t$ be unit vectors in $\bR^n$ such that $\ipr{w_i}{w_j}=\epsi$ for all $i,j$ with $1\leq i < j \leq t$, where $\epsi\in [0,1)$.
Then the set $\{w_1,\dots,w_t\}$ can be extended to $\{w_1,\dots,w_n\}$ such that $\ipr{w_i}{w_j}=\epsi$ for all $i,j$  with $1\leq i < j \leq n$, and such that for some orthonormal basis $e_1,\dots,e_n$ we have \[w_i = \frac{e_i+\lambda e}{\norm{e_i+\lambda e}}\quad (i=1,\dots,n),\]
where \[\lambda \colonequals \frac{-1+\sqrt{1+\epsi n/(1-\epsi)}}{n} \quad\text{and}\quad e \colonequals \sum_{j=1}^n e_j = \frac{1}{\sqrt{1+(n-1)\epsi}}\sum_{j=1}^n w_j.\]
Moreover, $\norm{e_i+\lambda e}^2 = (1-\epsi)^{-1}$ for each $i\in \{1,\dots,n\}$ and for every $x\in\bR^n$ we have
\[ \sum_{j=1}^n (\ipr{x}{w_j}-\epsi)^2 = (1-\epsi)(\norm{x}^2-\epsi) + \epsi\left(\ipr{x}{e}-\sqrt{1+(n-1)\epsi}\right)^2.\]
\end{lemma}
\begin{proof}
Let $\epsi\in [0,1)$ and $w_1,\dots,w_t$ be from the statement of the lemma.
We first extend the set $\{w_1,\dots,w_t\}$ to a set $\{w_1,\dots,w_n\}$ of unit vectors in $\bR^n$ so that $\ipr{w_i}{w_j}=\epsi$ for all $i,j$  with $1\leq i < j \leq n$.
We proceed iteratively, choosing $w_{i+1}$ after the vectors $w_1,\dots,w_i$ have been obtained for some $i$ with $t \leq i < n$. 
The condition $\ipr{w_i}{w_j} = \epsi$ for each $j$ with $1 \leq j < i$ says that the desired point $w_{i+1}$ lies in the hyperplanes $\setbuilder{x \in \bR^n}{\ipr{x}{w_j} = \epsi}$.
Since $\epsi \geq 0$, the intersection of these hyperplanes taken over $j \in \{1,\dots,i\}$ is an affine subspace~$A$ of dimension at least $n-i \geq 1$.
The subspace~$A$ contains the point 
\begin{align*}
a& \colonequals \frac{\epsi}{1+(i-1)\epsi} \cdot (w_1+\dots+w_i) \\
&= \frac{\epsi}{1+\epsi(i-1)}(w_1+\dots+w_{j-1}+w_{j+1}+\dots+w_i) + \frac{1}{1+\epsi(i-1)}\epsi w_j,
\end{align*}
which is in the convex hull of $\{w_1,\dots,w_{j-1},\epsi w_j,w_{j+1},\dots,w_i\}$ for every $j\in \{1,\dots,i\}$.
Since $\epsi < 1$, each point $\epsi w_j$ is inside the unit ball centred in the origin and so is~$a$.
Additionally, $A$ does not contain any of the points $w_1,\dots,w_i$.
Altogether, $A$ intersects the unit sphere centred in the origin at a point that is not in $\{w_1,\dots,w_i\}$.
We let $w_{i+1}$ be an arbitrary point from this intersection.
For $i=n-1$, we obtain the set $\{w_1,\dots,w_n\}$.

We now let $e_i \colonequals w_i/\sqrt{1-\epsi}-\lambda e$ for every $i \in \{1,\dots,n\}$, where $e=\frac{1}{\sqrt{1+(n-1)\epsi}}\sum_{j=1}^n w_j$ and $\lambda = (-1+\sqrt{1+\epsi n/(1-\epsi)})/n$.
That is, we skew the vectors $w_1,\dots,w_n$ so that they are pairwise orthogonal and we scale the resulting vectors so that they will form an orthonormal basis.
Note that $\lambda n+1 = \sqrt{\frac{1+(n-1)\epsi}{1-\epsi}}$.
Using this fact and the choice of $e_1,\dots,e_n$, we obtain
\[\sum_{j=1}^n e_j = \frac{\sum_{j=1}^n w_j }{\sqrt{1-\epsi}}- \lambda n e = \left(\sqrt{\frac{1+(n-1)\epsi}{1-\epsi}} - \lambda n\right) e = e.\]

We now verify that $e_1,\dots,e_n$ form an orthonormal basis of $\bR^n$.
Let $i \in \{1,\dots,n\}$ be fixed.
Note that $\ipr{w_i}{e} = \frac{1+(n-1)\epsi}{\sqrt{1+(n-1)\epsi}}=\sqrt{1+(n-1)\epsi}$, as every $w_j$ is a unit vector and pairwise inner products of vectors $w_1,\dots,w_n$ equal $\epsi$.
Summing over $i$, we obtain $\ipr{e}{e} = \frac{1}{\sqrt{1+(n-1)\epsi}}\sum_{i=1}^n\ipr{w_i}{e} = n$.
Using these facts and the choice of $\lambda$, we derive
\begin{align*}\norm{e_i}^2 &= \ipr{\frac{w_i}{\sqrt{1-\epsi}}-\lambda e}{\frac{w_i}{\sqrt{1-\epsi}}-\lambda e} 
= \frac{\norm{w_i}^2}{1-\epsi} - 2\lambda\frac{\ipr{w_i}{e}}{\sqrt{1-\epsi}}  + \lambda^2\norm{e}^2 \\
&= \frac{1}{1-\epsi} - 2\lambda\sqrt{\frac{1+(n-1)\epsi}{1-\epsi}}+\lambda^2 n 
= \frac{1}{1-\epsi} - 2\lambda(\lambda n + 1) + \lambda^2n  \\
&= \frac{1}{1-\epsi} - \lambda^2n - 2\lambda = 1.
\end{align*}
Therefore each $e_i$ is a unit vector.
Similarly, for all distinct $i$ and $j$ from $\{1,\dots,n\}$, we have
\begin{align*}\ipr{e_i}{e_j} =& \ipr{\frac{w_i}{\sqrt{1-\epsi}}-\lambda e}{\frac{w_j}{\sqrt{1-\epsi}}-\lambda e} 
= \frac{\ipr{w_i}{w_j}}{1-\epsi} - \lambda\frac{\ipr{w_i}{e}}{\sqrt{1-\epsi}} - \lambda\frac{\ipr{w_j}{e}}{\sqrt{1-\epsi}}  + \lambda^2\norm{e}^2 \\
&= \frac{\epsi}{1-\epsi} - 2\lambda\sqrt{\frac{1+(n-1)\epsi}{1-\epsi}}+\lambda^2 n 
= \frac{\epsi}{1-\epsi} - \lambda^2n - 2\lambda
= 0.
\end{align*}
We thus see that $e_1,\dots,e_n$ is indeed an orthonormal basis in $\bR^n$.

To show $\norm{e_i+\lambda e}^2 = (1-\epsi)^{-1}$ for each $i\in \{1,\dots,n\}$, we simply use the fact that $w_i$ is a unit vector and derive
\[\norm{e_i + \lambda e}^2 = \ipr{\frac{w_i}{\sqrt{1-\epsi}}}{\frac{w_i}{\sqrt{1-\epsi}}} 
= \frac{\norm{w_i}^2}{1-\epsi}  = \frac{1}{1-\epsi}.\]

It remains to prove the last expression in the statement of the lemma.
Let $x$ be an arbitrary point from $\bR^n$.
Since $e_1,\dots,e_n$ is a basis of $\bR^n$, we have $x=\sum_{i=1}^n \alpha_i e_i$ for some $(\alpha_1,\dots,\alpha_n) \in \bR^n$.
For each $j \in \{1,\dots,n\}$, we express the term $\ipr{x}{w_j}$ as 
\[\ipr{x}{w_j} 
= \sum_{i=1}^n \alpha_i\ipr{e_i}{w_j} 
= \sqrt{1-\epsi}\sum_{i=1}^n \alpha_i(\ipr{e_i}{e_j} + \lambda\ipr{e_i}{e})
= \sqrt{1-\epsi}(\alpha_j + \lambda\ipr{x}{e}),\]
using the facts that the basis $e_1,\dots,e_n$ is orthonormal and that $w_j=\sqrt{1-\epsi}(e_j+\lambda e)$.
Now, we have
\begin{align*}
&\sum_{j=1}^n \bigl(\ipr{x}{w_j} - \epsi\bigr)^2 
= \sum_{j=1}^n \bigl(\sqrt{1-\epsi}(\alpha_j + \lambda\ipr{x}{e}) - \epsi\bigr)^2 \\
&= (1-\epsi)\sum_{j=1}^n \bigl(\alpha_j + \lambda\ipr{x}{e}\bigr)^2 -2\epsi\sqrt{1-\epsi}\biggl(\sum_{j=1}^n\alpha_j + \lambda\ipr{x}{e}\biggr)+n\epsi^2 \\
&= (1-\epsi)\biggl(\sum_{j=1}^n \alpha_j^2 + 2\lambda\ipr{x}{e}\sum_{j=1}^n \alpha_j + \lambda^2 n \ipr{x}{e}^2\biggr) - 2\epsi\sqrt{1-\epsi}\biggl(\sum_{j=1}^n\alpha_j + \lambda n\ipr{x}{e}\biggr) + n\epsi^2.
\end{align*}

Since the basis $e_1,\dots,e_n$ is orthonormal, we have $\norm{x}^2 = \sum_{i=1}^n\sum_{j=1}^n \alpha_i\alpha_j \ipr{e_i}{e_j} = \sum_{i=1}^n\alpha_i^2$ and, since $e=\sum_{j=1}^n e_j$, we also have $\ipr{x}{e} = \sum_{i=1}^n \sum_{j=1}^n \alpha_i\ipr{e_i}{e_j} = \sum_{i=1}^n\alpha_i$.
The above expression thus equals
\[(1-\epsi)\bigl(\norm{x}^2 + (2\lambda+\lambda^2n)\ipr{x}{e}^2\bigr) - 2\epsi\sqrt{1-\epsi}(1+\lambda n)\ipr{x}{e} + n\epsi^2.\]
Using the facts that $2\lambda+\lambda^2n = \epsi/(1-\epsi)$ and $1+\lambda n = \sqrt{\frac{1+(n-1)\epsi}{1-\epsi}}$, this expression can be further simplified as
\[(1-\epsi)\norm{x}^2 + \epsi\left(\ipr{x}{e}^2 - 2\sqrt{1+(n-1)\epsi}\ipr{x}{e} + n\epsi\right)\]
and then rewritten to the final form
\[(1-\epsi)(\norm{x}^2-\epsi) + \epsi\left(\ipr{x}{e}-\sqrt{1+(n-1)\epsi}\right)^2.\qedhere\]
\end{proof}

We are now ready to prove Theorem~\ref{thm1}.
For $d \ge 2$, let $V\subset\bR^d$ be an almost-equidistant set.
We let $G=(V,E)$ be the unit-distance graph of $V$ and let $k \colonequals \lfloor 2\sqrt{d}\rfloor$.
Note that $1 \leq k \leq d$. 

Let $S\subseteq V$ be a set of $k$ points such that the distance between any two of them is $1$.
If such a set does not exist, then, since the complement of $G$ does not contain a triangle, we have $|V| < R(k,3)$, where $R(k,3)$ is the Ramsey number of $K_k$ and $K_3$.
Using the bound $R(k,3) \leq \binom{k+3-2}{3-1}$ obtained by Erd\H{o}s and Szekeres~\cite{ES1935}, we derive $|V| < \binom{2\sqrt{d}+1}{2} = 2d+\sqrt{d}$.
Thus we assume in the rest of the proof that $S$ exists.

Let $B$ be the set of common neighbours of $S$, that is,
\[ B \colonequals \setbuilder{x\in V}{\norm{x-s}=1\text{ for all }s\in S}.\]
Since $V$ is equidistant, 
the set of non-neighbours of any vertex of $G$ is a clique and so it has size at most $d+1$ by Corollary~\ref{corollary:clique}.
Every vertex from $V\setminus B$ is a non-neighbour of some vertex from~$S$ and thus it follows that $\card{V\setminus B} \leq k(d+2)$.

We now estimate the size of $B$.
By Lemma~\ref{lemma:sphere} applied to $S$, the set $B$ lies on a sphere of radius $\sqrt{(k+1)/2k}$ in an affine subspace of dimension $d-k+1$.
We may take the centre of this sphere as the origin, and rescale by $\sqrt{2k/(k+1)}$ to obtain a set $B'$ of $m$ unit vectors $v_1,\dots,v_m\in\bR^{d-k+1}$ where $m \colonequals \card{B}$.
For any three of the vectors from~$B'$, the distance between some two of them is $\sqrt{2k/(k+1)}$.
For two such vectors $v_i$ and $v_j$ with $\norm{v_i-v_j}^2 = 2k/(k+1)$, the facts $\norm{v_i-v_j}^2 = \norm{v_i}^2+\norm{v_j}^2-2\ipr{v_i}{v_j}$ and $\norm{v_i}^2=\norm{v_j}^2=1$ imply $\ipr{v_i}{v_j} = \epsi$, where $\epsi \colonequals 1/(k+1)$. 
Note that the opposite implication also holds.
That is, if $\ipr{v_i}{v_j} = \epsi$, then $v_i$ and $v_j$ are at distance $\sqrt{2k/(k+1)}$.

Let $A=[a_{i,j}]$ be the $m\times m$ matrix defined by $a_{i,j} \colonequals \ipr{v_i}{v_j} - \epsi$.
Clearly, $A$ is a symmetric matrix with real entries.
If $m \ge d-k+2$, then $A$ is also non-zero, as $G$ contains no $K_{d+2}$ and every vertex from $B$ is adjacent to every vertex from $S$ in $G$.
We recall that $\rank XY \leq \min\{\rank X, \rank Y\}$ and $\rank (X+Y) \leq \rank X + \rank Y$ for two matrices $X$ and $Y$.
Since $B' = \{v_1,\dots,v_m\} \subset \bR^{d-k+1}$ and 
\[A = \begin{bmatrix} v_1 & v_2 & \cdots & v_m\end{bmatrix}^\top \begin{bmatrix} v_1 & v_2 & \cdots & v_m\end{bmatrix} - \epsi J,\] 
where $J$ is the $m\times m$ matrix with each entry equal to $1$, we have
\begin{equation}\label{rank1}
\rank A\leq d-k+2.
\end{equation}
By Lemma~\ref{lemma:rank},
\begin{equation}\label{rank2}
\rank A \geq \frac{\left(\sum\limits_{i=1}^m a_{i,i}\right)^2}{\sum\limits_{i=1}^m\sum\limits_{j=1}^m a_{i,j}^2} = \frac{m^2(1-\epsi)^2}{\sum\limits_{i=1}^m \sum\limits_{j=1}^m (\ipr{v_i}{v_j} - \epsi)^2}.
\end{equation}
For $i \in \{1,\dots,m\}$, let $N_i$ be the set of vectors from $B'$ that are at distance $\sqrt{2k/(k+1)}$ from~$v_i$.
That is,
\[ N_i \colonequals \setbuilder{v_j \in B'}{\ipr{v_i}{v_j} = \epsi}.\]
Then for each fixed $v_i$ we have
\begin{equation}\label{ineq1}
\sum_{j=1}^m (\ipr{v_i}{v_j}-\epsi)^2 = (1-\epsi)^2 + \sum_{v_j\in N_i} 0 + \sum_{v_j \in B' \setminus (N_i \cup \{v_i\})} (\ipr{v_i}{v_j}-\epsi)^2.
\end{equation}

Note that the vectors from $B' \setminus (N_i \cup \{v_i\})$ have pairwise inner products $\epsi$, as neither of them is at distance $\sqrt{2k/(k+1)}$ from $v_i$, and thus $|B' \setminus (N_i \cup \{v_i\})| \leq d-k+2$.
In fact, we even have $|B' \setminus (N_i \cup \{v_i\})| \leq d-k+1$, since $B'$ contains only unit vectors and any subset of $d-k+2$ points from $B'$ with pairwise distances $\sqrt{2k/(k+1)}$ would form the vertex set of a regular $(d-k+1)$-simplex with edge lengths $\sqrt{2k/(k+1)}$ centred at the origin.
However, then the distance from the centroid of such a simplex to its vertices would be equal to $\sqrt{k(d-k+1)/((k+1)(d-k+2))} \neq 1$, which is impossible.

Thus setting $n\colonequals d-k+1$ and $t \colonequals |B' \setminus (N_i \cup \{v_i\})|$, we have $t \leq n$.
Applying Lemma~\ref{lemma:calc} to the $t$ vectors from $B' \setminus (N_i \cup \{v_i\}) \subseteq \bR^n$ with $\epsi=(k+1)^{-1}$ and $x=v_i$, we see that the last sum in~\eqref{ineq1} is at most
\[
(1-\varepsilon)^2 + \varepsilon\left(\ipr{v_i}{e}-\sqrt{1+(d-k)\epsi}\right)^2,
\]
where $e=\sum_{j=1}^{d-k+1}e_j$ for some orthonormal basis $e_1,\dots,e_{d-k+1}$ of $\bR^{d-k+1}$.

By the Cauchy--Schwarz inequality,
\begin{align*}
\left(\ipr{v_i}{e}-\sqrt{1+(d-k)\epsi}\right)^2 &< \left(\sqrt{d-k+1}+\sqrt{1+(d-k)\epsi}\right)^2\\
&= d-k+1 + 2\sqrt{d-k+1}\sqrt{1+(d-k)\epsi} + 1+(d-k)\epsi\\
&< 4 (d-k+1).
\end{align*}

Thus, using $\epsi = (k+1)^{-1}$, we obtain
\begin{align*}
\sum_{j=1}^m \bigl(\ipr{v_i}{v_j}-\epsi\bigr)^2 & < 2(1-\epsi)^2 + 4\epsi(d-k+1)\\
&= 4\epsi d + 2(1+\epsi)^2 - 4 < 4\epsi d.
\end{align*}
If we substitute this upper bound back into \eqref{rank2}, then with \eqref{rank1} we obtain that $d-k+2 > m^2(1-\epsi)^2/(4m\epsi d)$ and thus $m < (4\epsi d)(d-k+2)/(1-\epsi)^2$.
Using the choice $k = \lfloor 2\sqrt{d}\rfloor$ and the expression $\epsi = (k+1)^{-1}$, we obtain $(d-k+2)/(1-\epsi)^2 < d$, if $d \geq 8$, and thus $m < 4d^2/(k+1)$.

Altogether, we have $m \leq \max\{d-k+1,4d^2/(k+1)\} = 4d^2/(k+1)$.
It follows that $\card{V} \leq k(d+2) + 4d^2/(k+1)$.
Again, using the choice $k= \lfloor 2\sqrt{d}\rfloor \in (2\sqrt{d}-1,2\sqrt{d}]$, we conclude that \[\card{V} < 2\sqrt{d}(d+2) + 4d^2/(2\sqrt{d}) = 4d^{3/2} + 4\sqrt{d}. 
\]
This finishes the proof of Theorem~\ref{thm1}.

\section{Low dimensions}
\label{section:lowd}

In this section, we give proofs of Theorems~\ref{thm:2d}, \ref{thm:3d}, \ref{thm:4d}, and \ref{thm:5d}.
Before doing so, we introduce the notion of abstract almost-equidistant graphs.
Denote the complete $t$-partite graph with classes of sizes $m_1,\dots,m_t$ by $K_{m_1,\dots,m_t}$ or $K_t(m_1,\dots,m_t)$.
A graph $G$ is said to be an \emph{abstract almost-equidistant graph in $\bR^d$} if the complement of $G$ does not contain $K_3$ and either
\begin{itemize}
\item 
$d=2$ and $G$ does not contain $K_4$ nor $K_{2,3}$;
\item 
or $d\geq 3$, $d$ odd and $G$ does not contain $K_{d+2}$ nor $K_{(d+1)/2}(3,\dots,3)$;
\item 
or $d\geq 4$, $d$ even and $G$ does not contain $K_{d+2}$ nor $K_{(d+2)/2}(1,3,\dots,3)$.
\end{itemize}

The following lemma justifies the notion of abstract almost-equidistant graphs.
We will see later that its converse is not true, as there are abstract almost-equidistant graphs in $\bR^d$ that are not unit-distance graphs of any point set from $\bR^d$.

\begin{lemma}
\label{lemma:abstractGraphs}
For every $d \geq 2$ and every almost-equidistant set $P$ from $\bR^d$, the unit-distance graph of $P$ is an abstract almost-equidistant graph in $\bR^d$.
\end{lemma}
\begin{proof}
Let $G$ be the unit-distance graph of an almost-equidistant set $P \subset\bR^d$.
Clearly, the complement of $G$ does not contain a triangle, as $P$ is almost equidistant.
The graph $G$ also does not contain a copy of $K_{d+2}$ by Corollary~\ref{corollary:clique}.
Thus it remains to show that $K_{2,3}$ does not occur as a unit-distance graph in $\bR^2$, $K_{(d+1)/2}(3,\dots,3)$ does not occur as a unit-distance graph for odd $d\geq 3$, and $K_{(d+2)/2}(1,3,\dots,3)$ does not occur as a unit-distance graph for even $d\geq 4$.

For $d=2$, Suppose that $K_{2,3}$ occurs as a unit-distance graph in $\bR^2$.
Let the class with two points be $\{p,q\}$.
The set of points that are at the same distance from both $p$ and $q$ is the intersection of two unit circles centred at $p$ and $q$, respectively, and thus contains at most two points.
Therefore, the other class cannot contain three points.

For odd $d\geq 3$, suppose for contradiction that there exist sets $V_1,\dots,V_k$, where $k=(d+1)/2$, such that each $V_i$ contains three points, and such that the distance between any two points from different $V_i$ equals $1$.
As in the proof of Lemma~\ref{lemma:sphere}, it is easy to show that each $V_i$ lies on a circle.
Furthermore, for any distinct $i,j$, $a_1-a_2\perp b_1-b_2$ for all $a_1,a_2\in V_i$ and $b_1,b_2\in V_j$.
Thus the affine hulls of the $V_i$s are pairwise orthogonal to each other, hence $V_1\cup \dots\cup V_k$ together span a space of dimension at least $2k > d$, a contradiction.

For even $d \ge 4$, suppose that there exist sets $V_1,\dots,V_k$ in $\bR^d$, where $k=d/2$, such that $V_1$ contains four points $a,b,c,d$ with $\norm{a-b}=\norm{a-c}=\norm{a-d}=1$, the sets $V_2,\dots,V_k$ each contains three points, and such that the distance between any two points from different $V_i$ equals $1$.
As in the case of odd $d$, the $V_i$ lie on circles (or a sphere in the case of $V_1$), the affine hulls of the $V_i$s have dimension at least $2$ and are pairwise orthogonal to each other, hence $V_1\cup \dots\cup V_k$ together span a space of dimension at least $2k = d$.
It follows that the affine hull of $V_1\cup \dots\cup V_k$ has dimension exactly $2k$, and that the affine hulls of all the the $V_i$s are $2$-dimensional.
In particular, we have that $a,b,c,d$ lie on a circle $C$. However, then $C$ and the unit circle with centre $a$ intersect in $3$ points, a contradiction.
\end{proof}

An abstract almost-equidistant graph $G=(V,E)$ in $\bR^d$ 
is \emph{realisable \textup{(}in $\bR^d$\textup{)}} if there is a point set $P$ in $\bR^d$, called a \emph{realisation} of $G$, and a one-to-one correspondence $f \colon P \to V$ such that, for all points $p$ and $q$ from $P$, if $\{f(p),f(q)\} \in E$, then $p$ and $q$ are at unit distance.
If $G$ is not realisable in $\bR^d$, then we say that it is \emph{non-realisable}.
For a realisable graph $G$ and its realisation $P$, we sometimes do not distinguish between the vertices of~$G$ and the points from $P$.

By Lemma~\ref{lemma:abstractGraphs}, if there is no realisable abstract almost-equidistant graphs in $\bR^d$ on $n$ vertices, then there is no almost-equidistant set in $\bR^d$ of size $n$. 
Using a simple exhaustive computer search, we enumerated all non-isomorphic graphs that are abstract almost-equidistant in $\bR^d$ for $d \in \{2,3,4,5,6\}$.
We filtered out graphs that are \emph{minimal}, meaning that any graph obtained by removing any edge from such a graph is no longer abstract almost-equidistant.
We summarise our results obtained by the computer search in Tables~\ref{tab:abstractMinimal} and~\ref{tab:abstractAll}.
More detailed description of our computations can be found in Section~\ref{subsec:computations}.

The core of the proofs of Theorems~\ref{thm:2d}, \ref{thm:3d}, and \ref{thm:4d} is to show that none of the minimal abstract almost-equidistant graphs in $\bR^d$ is realisable in $\bR^d$ for $d=2,3,4$, respectively.

\begin{table}
\centering
\begin{tabular}{l|rrrrrr|r}
 $n$& $d=2$& $d=3$& $d=4$& $d=5$&    $d=6$& $\cdots$ & $K_3$-free complement    \\
 \hline
$ 4$&     2&     2&     2&     2&        2&        &             2\\   
$ 5$&     2&     3&     3&     3&        3&        &             3\\   
$ 6$&     2&     3&     4&     4&        4&        &             4\\   
$ 7$&     1&     4&     5&     6&        6&        &             6\\   
$ 8$&     1&     5&     8&     9&       10&        &            10\\   
$ 9$&     0&     5&    10&    14&       15&        &            16\\  
$10$&      &     4&    18&    25&       29&        &            31\\  
$11$&      &     1&    22&    46&       54&        &            61\\  
$12$&      &     0&    27&   106&      130&        &           147\\  
$13$&      &      &    12&   242&      339&        &           392\\  
$14$&      &      &     3&   653&     1052&        &          1274\\   
$15$&      &      &     1&  1946&     3969&        &          5036\\   
$16$&      &      &     1&  5828&    18917&        &         25617\\   
$17$&      &      &     0& 12654&   105238&        &        164796\\   
$18$&      &      &      &  8825&   655682&        &       1337848\\
$19$&      &      &      &   340&  3971787&        &      13734745\\
$20$&      &      &      &     8&  $\ge 1$&        &     178587364\\
$21$&      &      &      &     0&  $\ge 1$&        &    2911304940\\
$22$&      &      &      &      &  $\ge 1$&        &   58919069858\\
$23$&      &      &      &      &  $\ge 1$&        & 1474647067521\\
$24$&      &      &      &      &  $\ge 1$&        & ?\\
$25$&      &      &      &      &  $\ge 1$&        & ?\\
$26$&      &      &      &      &        ?&        & ?\\
$27$&      &      &      &      &        0&        & ?\\
\end{tabular}
\caption{Numbers of minimal abstract almost-equidistant graphs in $\bR^d$ 
on $n$ vertices for $d \in \{2,3,4,5,6\}$ and some values of $n$ from $\{4,\dots,27\}$.
For comparison, the last column contains numbers of minimal $n$-vertex graphs with triangle-free complements~\cite{Brandt, Brinkmann, oeis:trianglefreegraphs}.
The entries denoted by ``?'' are not known.}
\label{tab:abstractMinimal}
\end{table}

\begin{table}
\centering
\begin{tabular}{l|rrrrrr|r}
 $n$& $d=2$& $d=3$& $d=4$&          $d=5$&    $d=6$& $\cdots$ & $K_3$-free complement\\
 \hline
 $4$&     6&     7&     7&              7&        7&          &              7\\
 $5$&     7&    13&    14&             14&       14&          &             14\\
 $6$&     9&    29&    37&             38&       38&          &             38\\
 $7$&     2&    50&    97&            106&      107&          &            107\\
 $8$&     1&    69&   316&            402&      409&          &            410\\
 $9$&     0&    35&   934&           1817&     1888&          &           1897\\
$10$&      &     7&  2362&          11132&    12064&          &          12172\\
$11$&      &     1&  2814&          86053&   103333&          &         105071\\
$12$&      &     0&   944&         803299&  1217849&          &        1262180\\
$13$&      &      &    59&        7623096& 19170728&          &       20797002\\
$14$&      &      &     4&       58770989&        ?&          &      467871369\\
$15$&      &      &     1& $\le$305976655&        ?&          &    14232552452\\
$16$&      &      &     1&              ?&        ?&          &   581460254001\\
$17$&      &      &     0&              ?&        ?&          & 31720840164950\\
\end{tabular}
\caption{Numbers of all abstract almost-equidistant graphs  in $\bR^d$ 
on $n$ vertices for $d \in \{2,3,4,5,6\}$ and some values of $n$ from $\{4,\dots,17\}$.
For comparison, the last column contains numbers of $n$-vertex graphs with triangle-free complements~\cite{McKay, oeis:trianglefreegraphs}.
The entries denoted by ``?'' are not known.}
\label{tab:abstractAll}
\end{table}

\subsection{Proof of Theorem~\ref{thm:2d}}

We show that $f(2) = 7$ and that, up to congruence, there is only one planar almost-equidistant set with $7$ points, namely the Moser spindle.

A computer search shows that, up to isomorphism, there are exactly two abstract almost-equidistant graphs on $7$ vertices.
One of them is the Moser spindle (Figure~\ref{fig1}), which is clearly uniquely realisable in the plane up to congruence.
The other graph (part~(a) of Figure~\ref{fig-abstractD2}) is the graph of the square antiprism with one point removed, which is easily seen to be non-realisable in the plane.
Lemma~\ref{lemma:abstractGraphs} thus implies that the Moser spindle is the unique (up to congruence) almost-equidistant set in the plane on $7$ points.

\begin{figure}[ht]
\centering
\includegraphics[scale=1]{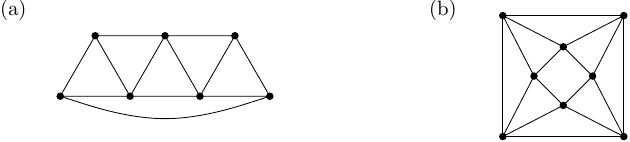}
\caption{Non-realisable abstract almost-equidistant graphs in $\bR^2$ on $7$ and $8$ vertices.}
\label{fig-abstractD2}
\end{figure}

There is a unique abstract almost-equidistant graph on $8$ vertices, namely the graph of the square antiprism (part~(b) of Figure~\ref{fig-abstractD2}), which is not realisable in the plane because it contains a non-realisable subgraph, namely the non-realisable abstract almost-equidistant graph on $7$ vertices drawn in part~(a) of Figure~\ref{fig-abstractD2}.
Thus, by Lemma~\ref{lemma:abstractGraphs}, there is no almost-equidistant set in the plane on $8$ points.
 
\subsection{Proof of Theorem~\ref{thm:3d}}

We prove $f(3)=10$ and that, up to congruence, there is only one almost-equidistant set in $\bR^3$ with $10$ points.

A computer search shows that there is exactly one abstract almost-equidistant graph $G_{11}$ in $\bR^3$ on $11$ vertices (Figure~\ref{fig-abstractD3_11}), and exactly $7$ abstract almost-equidistant graphs in~$\bR^3$ on $10$ vertices, four of which are minimal (Figures~\ref{fig2} and \ref{fig-abstractD3_10}).

\begin{figure}[ht]
\centering
\includegraphics[scale=1.1]{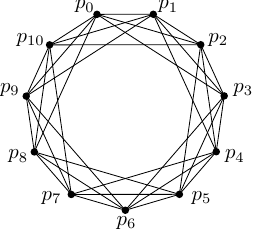}
\caption{The non-realisable abstract almost-equidistant graph $G_{11}$ in $\bR^3$ on $11$ vertices.}
\label{fig-abstractD3_11}
\end{figure}

Although $G_{11}$ contains the first graph in Figure~\ref{fig-abstractD3_10}, and we will show that none of the graphs in Figure~\ref{fig-abstractD3_10} are realisable in $\bR^3$, it is worth observing the following direct argument for the non-realisability of $G_{11}$.
Suppose for contradiction that it is realisable in $\bR^3$.
We label the vertices $p_0$ to $p_{10}$, with subscripts modulo $11$,
such that $\{p_i,p_j\}$ is an edge of $G_{11}$ if and only if $i-j\equiv \pm 1,\pm 2\pmod{11}$.
There are $11$ equilateral tetrahedra $p_i p_{i+1}p_{i+2}p_{i+3}$ in a realisation of $G_{11}$.
Let $T\colon\bR^3\to\bR^3$ be the unique isometry that maps the tetrahedron $p_0p_1p_2p_3$ to $p_1p_2p_3p_4$, that is, $T(p_i)=p_{i+1}$ for $i=0,1,2,3$.
Note that the vertex $p_{i+4}$ is uniquely determined by $p_i p_{i+1}p_{i+2}p_{i+3}$.
In fact $p_{i+4}$ is the reflection of $p_i$ through the centroid of the triangle $p_{i+1}p_{i+2}p_{i+3}$.
It follows that $T(p_i)=p_{i+1}$ for every $i$.
Therefore the centroid $c \colonequals \frac{1}{11}\sum_{i=0}^{10}p_i$ is a fixed point of $T$ and all points $p_i$ are on a sphere with centre $c$.
However, the points $p_0,p_1,p_2,p_3,p_4$ are easily seen not to lie on a sphere (for example, by using Lemma~\ref{lemma:sphere}) and we have a contradiction.
Thus, by Lemma~\ref{lemma:abstractGraphs}, there is no almost-equidistant set in $\bR^3$ on $11$ points.

\begin{figure}[ht]
\centering
\includegraphics[scale=1.05]{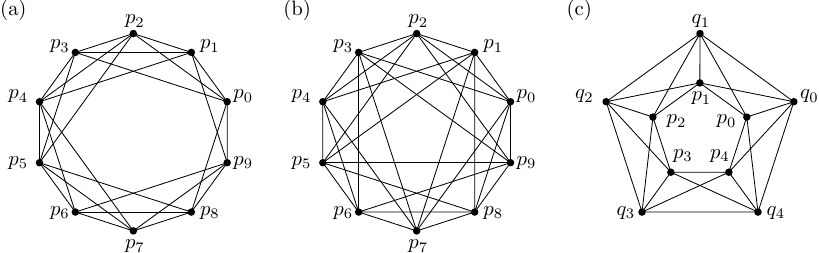}
\caption{Non-realisable abstract almost-equidistant graphs in $\bR^3$ on $10$ vertices.}
\label{fig-abstractD3_10}
\end{figure}

We have already described a realisation in $\bR^3$ of the graph from Figure~\ref{fig2} in the proof of Theorem~\ref{thm:lowerbound}.
It is unique up to congruence, and it is easy to check that there are no further unit distances between vertices.
We next show that the other three graphs are not realisable in $\bR^3$.

The graph in part~(a) of Figure~\ref{fig-abstractD3_10} consists of six copies of $K_4$, namely $p_0p_1p_2p_3$, $p_1p_2p_3p_4$, $p_2p_3p_4p_5$, $p_4p_5p_6p_7$, $p_5p_6p_7p_8$, $p_6p_7p_8p_9$, together with the edges $p_0p_9$, $p_1p_9$ and $p_0p_8$.
We may choose coordinates $p_4=(0,0,1/2)$, $p_5=(0,0,-1/2)$, $p_2=(\sqrt{3}/2,0,0)$, $p_3=(1/(2\sqrt{3}), \sqrt{2/3},0)$, and then we will have $p_7=(\frac{\sqrt{3}}{2}\cos\beta,\frac{\sqrt{3}}{2}\sin\beta,0)$ for some angle $\beta$, and $p_6=(\frac{\sqrt{3}}{2}\cos(\beta+\alpha),\frac{\sqrt{3}}{2}\sin(\beta+\alpha),0)$, where $\cos\alpha=1/3$.
Thus $p_6=(\frac{1}{2\sqrt{3}}\cos\beta\mp\frac{\sqrt{2}}{\sqrt{3}}\sin\beta, \frac{1}{2\sqrt{3}}\sin\beta\mp\frac{\sqrt{2}}{\sqrt{3}}\cos\beta, 0)$.
It is now simple to determine the coordinates of the remaining points by taking reflections.
In particular, we obtain $p_1 = (\frac{4}{3\sqrt{3}}, \frac{2\sqrt{2}}{3\sqrt{3}},\frac{5}{6})$ and $p_9 = (\frac{20}{9\sqrt{3}}\cos\beta\mp\frac{10\sqrt{2}}{9\sqrt{3}}\sin\beta,\frac{20}{9\sqrt{3}}\sin\beta\pm\frac{10\sqrt{2}}{9\sqrt{3}}\cos\beta,-\frac{1}{18})$.
If we now calculate the distance $\norm{p_1-p_9}$, we obtain either $\sqrt{\frac{112}{27}-\frac{80}{27}\cos\beta}$ or $\sqrt{\frac{112}{27}-\frac{80}{81}\cos\beta - \frac{160\sqrt{2}}{81}\sin\beta}$, depending on the sign of $\sin\alpha$.
However, both these expressions are larger than $1$.

The graph in part~(b) of Figure~\ref{fig-abstractD3_10} is the complement of the Petersen graph, and contains $5$ copies of the graph of the regular octahedron, one for every induced matching of three edges in the Petersen graph.
Each such octahedron is uniquely realisable in~$\bR^3$.
The octahedron $O_1$ with diagonals $\{p_0p_6, p_7p_9, p_3p_8\}$ and the octahedron $O_2$ with diagonals $\{p_0p_4, p_1p_3, p_2p_7\}$ have a common face $p_0p_3p_7$.
The only way to realise $O_1$ and $O_2$ is with the opposite faces $p_6p_8p_9$ and $p_4p_1p_2$ in two planes parallel to the plane of $p_0p_3p_7$.
Since an octahedron with edges of unit length has width $\sqrt{2/3}$,  we obtain that the distance between $p_4$ and $p_6$ (both opposite to $p_0$ in $O_1$ and $O_2$, respectively) is $2\sqrt{2/3}$, which is impossible, as they have to be at unit distance.

The graph in part~(c) of Figure~\ref{fig-abstractD3_10} is a ring of $5$ unit tetrahedra $p_i q_i p_{i+1}q_{i+1}$, $i=0,\dots,4$ with indices taken modulo $5$, with two successive ones joined at an edge $p_iq_i$.
Suppose for contradiction that we have a realisation of this graph in $\bR^3$.
Let $m_i$ be the midpoint of $p_iq_i$ for each $i \in \{0,\dots,4\}$.
The distance between $m_i$ and $m_{i+1}$ is $r=1/\sqrt{2}$.
Since each $p_iq_i$ is orthogonal to $p_{i+1}q_{i+1}$, it follows that $m_i$, $p_{i-1}q_{i-1}$, and $p_{i+1}q_{i+1}$ lie in the same plane and the lines $p_{i-1}q_{i-1}$ and $p_{i+1}q_{i+1}$ are tangent to the circle $C_i$ in this plane with centre $m_i$ and radius $r$.
If $C_1$ and $C_3$ were in the same plane, then the edges $p_4q_4$ and $p_0q_0$ of the tetrahedron $p_0q_0p_4q_4$ would have to be coplanar.
Therefore $C_1$ and $C_3$ are in different planes, but they have the same tangent line $p_2q_2$ touching both at $m_2$.
It follows that they lie on a unique sphere $\Sigma$.
The lines $p_0q_0$ and $p_4q_4$ are tangent to $\Sigma$ at $m_0$ and $m_4$, respectively.
It follows that the plane $\Pi_0$ through $m_0$ orthogonal to $p_0q_0$ and the plane $\Pi_4$ through $m_4$ orthogonal to $p_4q_4$ both contain the centre $c$ of $\Sigma$.
Thus $c\in\Pi_0\cap\Pi_4$.
Since both $\Pi_0$ and $\Pi_4$ contain $m_0$ and $m_4$, it follows that $\Pi_0\cap\Pi_4$ is the line $m_0m_4$, and  it follows that $c\in m_0m_4$, that is, $\norm{m_0-m_4}=r$ is a diameter of $\Sigma$.
However, $\Sigma$ contains two circles on its boundary of radius $r$, a contradiction.

Thus there is only one realisable abstract almost-equidistant graph in $\bR^3$ on $10$ vertices and, by Lemma~\ref{lemma:abstractGraphs}, there is also a unique (up to congruence) almost-equidistant set in $\bR^3$ on $10$ points.

\subsection{Proof of Theorem~\ref{thm:4d}}

We want to show that $12 \le f(4) \le 13$.
To do so, we use the following result, which says that the graph in Figure~\ref{fig-abstractD4_10} is not realisable in $\bR^4$.

\begin{figure}[ht]
\centering
\includegraphics[scale=1.1]{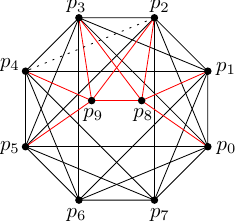}
\caption{A non-realisable abstract almost-equidistant graph $G_{10}$ in $\bR^4$ on $10$ vertices.
The missing edge $\{p_2,p_4\}$ is denoted by a dotted segment.}
\label{fig-abstractD4_10}
\end{figure}

\begin{lemma}\label{lemma:nonrealisable}
The graph $G_{10}$ in Figure~\ref{fig-abstractD4_10} is not realisable in $\bR^4$.
\end{lemma}
\begin{proof}
The graph $G_{10}$ is built up by starting with a cross-polytope with vertices $p_0,\dots,p_7$ and diagonals $p_0p_4$, $p_1p_5$, $p_2p_6$, $p_3p_7$, then removing edge $p_2p_4$, and then adding vertices $p_8$, $p_9$ and edges $p_8p_0$, $p_8p_1$, $p_8p_2$, $p_8p_3$, $p_8p_9$, $p_9p_2$, $p_9p_3$, $p_9p_4$, $p_9p_5$.

Suppose for contradiction that $G_{10}$ has a realisation in $\bR^4$.
We first show that the distance between $p_2$ and $p_4$ is necessarily $1$.
Note that $G_{10}$ contains a copy of $K_{4,4}$ with classes $V_1=\{p_1,p_3,p_5,p_7\}$ and $V_2=\{p_0,p_2,p_4,p_6\}$ as a subgraph.
Similarly as in the proof of Lemma~\ref{lemma:abstractGraphs}, it follows that $V_1$ and $V_2$ each lies on a circle with the same centre, which we take as the origin, in orthogonal $2$-dimensional planes.
The induced subgraph $G_{10}[p_1,p_3,p_5,p_7]$ of $G_{10}$ is a $4$-cycle, hence the circle on which $V_1$ lies has radius $1/\sqrt{2}$.
Therefore the circle on which $V_2$ lies also has radius $1/\sqrt{2}$ and, since the induced subgraph $G_{10}[p_0,p_2,p_4,p_6]$ of $G_{10}$ is a path of length $3$, $p_0,p_2,p_4,p_6$ also have to be the vertices of a square and $\norm{p_2-p_4}=1$.

Therefore we have a cross-polytope with diagonals $p_ip_{i+4}$, $i=0,1,2,3$.
The graph of a cross-polytope can be realised in only one way in $\bR^4$ up to isometry.
Thus we may choose coordinates so that $p_0=-p_4=(1/\sqrt{2},0,0,0)$, $p_1=-p_5=(0,1/\sqrt{2},0,0)$, $p_2=-p_6=(0,0,1/\sqrt{2},0)$, $p_3=-p_7=(0,0,0,1/\sqrt{2})$.
Since $p_0p_1p_2p_3p_8$ is a clique, $p_8=(\lambda,\lambda,\lambda,\lambda)$, where $\lambda=(1\pm\sqrt{5})/(4\sqrt{2})$.
Since $p_2p_3p_4p_5p_9$ is a clique, we obtain similarly that $p_9=(-\mu,-\mu,\mu,\mu)$ where $\mu=(-1\pm\sqrt{5})/(4\sqrt{2})$.
However, then the distance $\norm{p_8-p_9}$ is one of the values $(\pm 1+\sqrt{5})/2$ or $\sqrt{3/2}$, but it has to equal $1$, a contradiction.
\end{proof}

Now, we can proceed with the proof of Theorem~\ref{thm:4d}.
The lower bound of $f(4) \ge 12$ follows from Theorem~\ref{thm:lowerbound}.
A computer search shows that there are no abstract almost-equidistant graphs in $\bR^4$ on $17$ or more vertices, a unique one on $16$ vertices, a unique one on $15$ vertices, and four on $14$ vertices, three of which are minimal; see Figures~\ref{fig-abstractD4_14_1} and~\ref{fig-abstractD4_14_2}.

The first of these three consists of the graphs of a $3$-dimensional octahedron $q_0\cdots q_5$ and a  $4$-dimensional cross-polytope $p_0\cdots p_7$ with a biregular graph between their respective vertex sets, as in part~(a) of Figure~\ref{fig-abstractD4_14_1}.
It contains the graph $G_{10}$ as a subgraph on vertices $\{p_0,\dots, p_7\}\cup\{q_0,q_1\}$.

\begin{figure}[ht]
\centering
\includegraphics[scale=1.05]{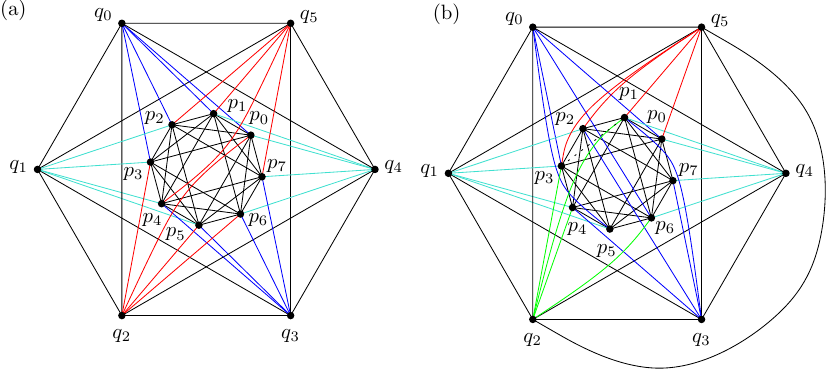}
\caption{Two non-realisable abstract almost-equidistant graphs in $\bR^4$ on $14$ vertices.
We use colours on edges between the two cross-polytopes to emphasise the symmetries of the graphs.
The missing edge $\{p_1, p_3\}$ in (b) is denoted by a dotted segment.}
\label{fig-abstractD4_14_1}
\end{figure}

The second of these graphs (part~(b) of Figure~\ref{fig-abstractD4_14_1}) contains the graph $G_{10}$ as an induced subgraph on vertices $\{p_0\cdots p_7\} \cup \{q_1,q_5\}$.
Note that the edge $\{p_1,p_3\}$ is missing.
Thus these two graphs are not realisable in $\bR^4$ by Lemma~\ref{lemma:nonrealisable}.

The last of the three graphs, called $G_{14}$ and shown in Figure~\ref{fig-abstractD4_14_2}, is the complement of the graph that is obtained from the cycle $C_{14}$ by adding the diagonals and chords of length $4$.
We show that $G_{14}$ is non-realisable in $\bR^4$ using an analogous approach as we  used to show that the graph $G_{11}$ is not realisable in $\bR^3$.

Suppose for contradiction that the graph $G_{14}$ is realisable in $\bR^4$.
We label the vertices $p_0$ to $p_{13}$, with subscripts modulo $14$,
such that $\{p_i,p_j\}$ is an edge of $G_{14}$ if and only if $i-j\equiv \pm 2,\pm 3, \pm 5,\pm 6\pmod{14}$.
There are $14$ equilateral $4$-simplices $p_i p_{i+2}p_{i+5}p_{i+8}p_{i+11}$ in a realisation of $G_{14}$.
Let $T\colon\bR^4\to\bR^4$ be the unique isometry that maps the simplex $p_0p_2p_5p_8p_{11}$ to $p_3p_5p_8p_{11}p_0$, that is, $T(p_i)=p_{i+3}$ for $i=0,2,5,8,11$.
Note that the vertex $p_{i+3}$ is the reflection of $p_{i+2}$ through the centroid of the tetrahedron $p_i p_{i+5}p_{i+8}p_{i+11}$ and thus it is uniquely determined by $p_i p_{i+5}p_{i+8}p_{i+11}$.
It follows that $T(p_i)=p_{i+3}$ for every~$i$.
Therefore the centroid $c \colonequals \frac{1}{14}\sum_{i=0}^{13}p_i$ is a fixed point of $T$ and all points $p_i$ are on a sphere with centre $c$.
However, the points $p_0,p_2,p_3,p_5,p_8,p_{11}$ are easily seen not to lie on a sphere (for example, by using Lemma~\ref{lemma:sphere}), a contradiction.

Using Lemma~\ref{lemma:abstractGraphs}, we conclude that every almost-equidistant set in $4$-space of maximum cardinality has at most $13$ points.

\begin{figure}[ht]
\centering
\includegraphics[scale=1]{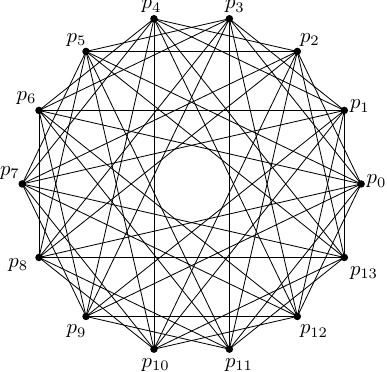}
\caption{The non-realisable abstract almost-equidistant graph $G_{14}$ in $\bR^4$ on $14$ vertices.}
\label{fig-abstractD4_14_2}
\end{figure}

\subsection{Proof of Theorem~\ref{thm:5d}}
Here, we prove the estimates $f(6)\geq 18$, $f(7)\geq 20$, and $f(9)\geq f(8)\geq 24$ using the construction of Larman and Rogers~\cite{Larman-Rogers} that gives $f(5) \geq 16$.
The computer search for the upper bounds is described in the next section.

We first briefly describe the Larman--Rogers construction of an almost-equidistant set of $16$ points in~$\bR^5$.
Let $V$ be the set of vertices of the cube $\{\pm 1\}^5$ in $\bR^5$ with an odd number of positive signs.
It is easy to check that $\card{V}=16$ and that for any three points in $V$, some two differ in exactly two coordinates.
Moreover, any two points differ in either two or four coordinates and then their distance is $\sqrt{8}$ or $4$, respectively.
It follows that $S \colonequals \frac{1}{\sqrt{8}}V$ is an almost-equidistant set in $\bR^5$.
Also note that the norm of every vector from $S$ is exactly $\sqrt{5/8}$.

Consider $\bR^6=E_1\oplus \bR e$ where $\dim E_1=5$ and $e$ is a unit vector orthogonal to $E_1$.
We place $S$ in $E_1$ and show that $S\cup\{\pm \sqrt{3/8} e\}$ is an almost-equidistant set of $18$ points in $\bR^6$.
Since $S$ is an almost-equidistant set in $\bR^5$, it suffices to check triples of points from $S\cup\{\pm \sqrt{3/8} e\}$ containing points from $\{\pm \sqrt{3/8} e\}$.
Let $T$ be such a triple.
Then $T$ contains a point $p$ from $S$ and $q$ from $\{\pm \sqrt{3/8} e\}$.
Since $p$ and $q$ are orthogonal, we obtain $\norm{p-q}^2 = \ipr{p}{p} - 2\ipr{p}{q} + \ipr{q}{q} = 5/8+0+3/8=1$.
Thus $T$ contains a pair of points at unit distance and, consequently, $f(6) \ge 18$.

Next consider $\bR^7=E_1\oplus E_2$, where $\dim E_1=5$, $\dim E_2=2$, and $E_1\perp E_2$.
We place~$S$ in~$E_1$ and $S'\colonequals\{(\pm\frac12,\pm\frac{1}{\sqrt{8}})\}$ in $E_2$.
We show that $S\cup S'$ is an almost-equidistant set of~$20$ points in $\bR^7$.
Again, since $S$ is an almost-equidistant set in $\bR^5$, we only need to check triples $T$ of points from $S\cup S'$ containing some of the vectors from $S'$.
Analogously as before, the distance between every point from $S'$ and every point of $S$ is $1$, as the norm of every vector from $S'$ is $\sqrt{3/8}$ and $E_1\perp E_2$.
We can thus assume $T \subseteq S'$.
If two vectors from $S'$ differ only in the first coordinate, then their distance is $1$.
Since every triple of vectors from $S'$ contains a pair of vectors that differ only in the first coordinate, we obtain that $T$ contains a pair of points at unit distance and thus $f(7) \geq 20$. 

Next consider $\bR^8=E_1\oplus E_3$ where $\dim E_1=5$ and $\dim E_3=3$.
We place $S$ in $E_1$ and $S' \colonequals \frac{1}{\sqrt{8}}\{\pm 1\}^3$ in $E_2$.
Again, it suffices to check triples $T$ of points from $S \cup S'$ with $T \cap S' \neq \emptyset$.
The distance between every point in $S$ and every point in $S'$ again equals $1$ and thus we can assume $T \subseteq S'$.
Every triple of points from $S'$ contains a pair of points that differ in exactly two coordinates and so they are at distance $1$.
Thus $S'$ is an almost-equidistant set of $8$ points.
It follows that $S\cup S'$ is an almost-equidistant set of $24$ points in~$\bR^8$, which implies that $f(9)\geq f(8)\geq 24$.

\subsection{The computer search}
\label{subsec:computations}

In this subsection we describe how we computed the entries as stated in Tables~\ref{tab:abstractMinimal} and~\ref{tab:abstractAll}, and also the upper bounds as stated in Table~\ref{table1}.

First, we describe our simple approach to generate all $n$-vertex abstract almost-equidistant graphs in~$\bR^d$ for given $n$ and $d$.
We start with a single vertex and repeatedly add a new vertex and go through all possibilities of joining the new vertex to the old vertices.
For each possibility of adding edges, we check if the resulting graph contains one of the two forbidden subgraphs and that its complement does not contain a triangle.
We use two tricks to speed this up.
First, when adding a vertex, we can assume that the newly inserted vertex has minimum degree in the extended graph.
Secondly, we only have to go through all possibilities of adding at least $n-d-1$ new edges, where $n$ is the number of vertices before extending the graph.
This is because the degree of the newly added vertex has to be at least $n-d-1$, since the complement of an abstract almost-equidistant graph $G$ is triangle-free, hence the non-neighbours of each vertex induce a clique in $G$, which has at most $d+1$ vertices.
To find all minimal graphs, we repeatedly attempt to remove an edge and check that the complement is still triangle-free.
Once this is no longer possible, we know that we have a minimal abstract almost-equidistant graph.

We implemented this approach in Sage and used it to obtain all abstract almost-equidistant graphs in $\bR^d$ for $d\in\{2,3,4\}$.
To find all abstract almost-equidistant graphs in $\mathbb{R}^5$ on at most $15$ vertices and in $\mathbb{R}^6$ on at most $13$ vertices, we used a C++ implementation of this approach.
The obtained numbers of abstract almost-equidistant graphs are summarized in Table~\ref{tab:abstractAll}.

Unfortunately, our program was not able to find all minimal abstract almost-equidistant graphs in $\mathbb{R}^d$ on $n$ vertices for $d \geq 5$ and large~$n$ in reasonable time.
To do this, we used the programs \emph{Triangleramsey}~\cite{triangleramsey:website, Brinkmann} by Brinkmann, Goedgebeur, and Schlage-Puchta and \emph{MTF}~\cite{mtf:website, Brandt} by Brandt, Brinkmann, and Harmuth.
These programs generate all minimal $K_{d+2}$-free graphs with no $K_3$ in their complement, the so-called \emph{Ramsey $(3,d+2)$-graphs}.
For each Ramsey $(3,d+2)$-graph on the output we tested whether it is a minimal abstract almost-equidistant graph in $\mathbb{R}^d$ using a simple C++ program that checks forbidden subgraphs from Lemma~\ref{lemma:abstractGraphs}.
This allowed us to find all minimal abstract almost-equidistant graphs in~$\bR^5$ and all minimal abstract almost-equidistant graphs in $\mathbb{R}^6$ with at most $19$ vertices; see Table~\ref{tab:abstractMinimal}.
We were also able to find some minimal abstract almost-equidistant graphs in~$\mathbb{R}^6$ on $25$ vertices.

To improve the upper bounds on~$f(d)$ for $d \in \{5,6,7\}$, we checked forbidden subgraphs in all minimal Ramsey $(3,7)$-graphs on $21$ and $22$ vertices, minimal Ramsey $(3,8)$-graphs on $27$ vertices, and minimal Ramsey $(3,9)$-graphs on $35$ vertices, respectively.
A complete list of these graphs is available on the website~\cite{McKay_ramsey3n_website} of McKay.
Since none of these graphs are abstract almost-equidistant, we obtain $f(5) \le 20$, $f(6) \le 26$, and $f(7) \le 34$; see Table~\ref{table1}.

The source code of our programs and the files are available on a separate  website~\cite{program_manfred}.

\section{A more general setting}
\label{section:general}

We consider the following natural generalization of the problem of determining the maximum sizes of almost-equidistant sets.
For positive integers $d$, $k$, and $l$ with $l \leq k$, let $f(d,k,l)$ be the maximum size of a point set $P$ in $\mathbb{R}^d$ such that among any $k+1$ points from $P$ there are at least $l+1$ points that are pairwise at unit distance.
Since every subset of $\mathbb{R}^d$ with all pairs of points at unit distance has size at most $d+1$, we have $f(d,1,1)=d+1$ for every $d$.
In the case $k=2$ and $l=1$, Theorems~\ref{thm:lowerbound} and~\ref{thm1} give $2d+4 \le f(d,2,1) \leq O(d^{3/2})$ for every $d \ge 3$.
In this section, we discuss the problem of determining the growth rate of $f(d,k,l)$ for larger values of $d$, $k$, and $l$.

A similar problem, where the notion of unit distance is replaced by orthogonality, has been studied by several authors~\cite{Alon-Szegedy, Deaett, Furedi-Stanley, Rosenfeld}.
More specifically, for positive integers $d$, $k$, and $l$ with $l \leq k$, let $\alpha(d,k,l)$ be the maximum size of a set $V$ of nonzero vectors from $\mathbb{R}^d$ such that among any $k+1$ vectors from $V$ there are at least $l+1$ pairwise orthogonal vectors.
F\"{u}redi and Stanley~\cite{Furedi-Stanley} showed that $\alpha(d,k,l) \le (1+o(1))\sqrt{\pi d/(2l)}((l+1)/l)^{d/2-1}k$.
Alon and Szegedy~\cite{Alon-Szegedy} used a probabilistic argument to show the following lower bound on $\alpha(d,k,l)$.

\begin{theorem}[\cite{Alon-Szegedy}]
\label{thm:AlonSzegedy}
For every fixed positive integer $l$ there are some $\delta=\delta(l)>0$ and $k_0(l)$ such that for every $k \ge k_0(l)$ and every $d \geq 2\log{k}$,
\[\alpha(d,k,l) \geq d^{\delta\log{(k+2)}/\log{\log{(k+2)}}},\]
where the logarithms are base 2.
\end{theorem}

Let $d$, $k$, and $l$ be positive integers with $d \geq 2\log{k}$ and with $k$ sufficiently large with respect to $l$.
It follows from the proof of Theorem~\ref{thm:AlonSzegedy} that there is a $\delta=\delta(l)>0$ and a subset $F=F(d,k,l)$ of $\{-1,1\}^d$ of size at least $d^{\delta\log{(k+2)}/\log{\log{(k+2)}}}$ such that among any $k+1$ vectors from $F$ there are at least $l+1$  pairwise orthogonal vectors.
We define the set $P_F=\frac{1}{\sqrt{2d}} \cdot F = \setbuilder{p_v=(v_1/\sqrt{2d},\dots,v_d/\sqrt{2d})}{v=(v_1,\dots,v_d) \in F}$.
Clearly, $|P_F|=|F|$.
Since $F \subseteq \{-1,1\}^d$, it is not difficult to verify that any two points $p_u$ and $p_v$ from $P_F$ are at unit distance if and only if  the vectors $u$ and $v$ from $F$ are orthogonal.
It follows that for every fixed positive integer $l$ there are some $\delta=\delta(l)>0$ and $k_0(l)$ such that for every $k \ge k_0(l)$ and every $d \geq 2\log{k}$,
\[f(d,k,l) \geq d^{\delta\log{(k+2)}/\log{\log{(k+2)}}}.\]

Let $d,k,l$ be positive integers.
The following simple argument, which is based on an estimate on the chromatic number of $\mathbb{R}^d$, gives an upper bound on $f(d,k,l)$ that is linear in $k$ and exponential in $d$.

Let $P$ be a set of points in $\mathbb{R}^d$ such that among any $k+1$ points from $P$ there are $l+1$ points that are pairwise at unit distance.
Let $G$ be the unit-distance graph for $P$.
Let $c$ be a colouring of $G$ with $m \colonequals \chi(G)$ colours and let $P=C_1 \cup \cdots \cup C_m$ be the colour classes  induced by~$c$.
For every unit-distance graph $H$ in $\mathbb{R}^d$, Larman and Rogers~\cite{Larman-Rogers} showed that $\chi(H) \le (3+o(1))^d$.
We thus have $m \le (3+o(1))^d$.
Since there are at least $l+1 \ge 1$ edges among any set of $k+1$ vertices of $G$, we have $|C_i| \le k$ for every $i \in \{1,\dots,m\}$.
In particular, $|P| \le mk \le (3+o(1))^d k$.
We thus obtain the following estimate.
\begin{proposition}
For any $k\geq 2$,
\[f(d,k,l) \le (3+o(1))^d k.\] 
\end{proposition}

\section*{Acknowledgements}
We thank Roman Karasev for explaining to us the argument that the graph in part~(c) of Figure~\ref{fig-abstractD3_10} is not realisable in $\bR^3$, Istv\'an Talata for drawing our attention to his results and those of Bernadett Gy\"orey, Dan Ismailescu for helpful conversations, Imre B\'ar\'any for his support, and an anonymous referee whose suggestions lead to an improved paper.


\end{document}